\documentclass[12pt,a4paper]{article}

\usepackage[utf8x]{inputenc}
\usepackage{amsmath}
\usepackage{amssymb}
\usepackage{graphicx}
\usepackage{color}

\usepackage{mathtools}
\usepackage{amsthm}

\numberwithin{equation}{section}

\newcommand{\eps}{\varepsilon}

\newcommand{\Z}{{\mathbb Z}}
\newcommand{\N}{{\mathbb N}}

\newcommand{\G}{{\mathcal G}}

\newcommand{\T}{{\mathcal T}}
\newcommand{\R}{{\mathbb R}}

\newcommand{\Q}{{\mathbb Q}}

\renewcommand{\phi}{\varphi}

\newcommand{\HH}{{\mathcal H}}

\renewcommand{\S}{\mathcal{S}}

\newcommand{\IE}{\mathbb{E}}
\newcommand{\IP}{\mathbb{P}}

\newcommand{\argmin}{\mathop{\text{arg$\,$min}}}

\allowdisplaybreaks

\newcommand{\F}{\mathcal{F}}

\newcommand{\1}[1]{{\mathbf{1}}{\{#1\}}}

\newtheorem{theo}{Theorem}[section]
\newtheorem{lem}[theo]{Lemma}

\newtheorem{rmk}[theo]{Remark}

\title{On large deviations for the cover time of two-dimensional torus}
\author{Francis Comets$^{1}$ \and Christophe Gallesco$^{2}$ \and
 Serguei~Popov$^{2}$ \and Marina Vachkovskaia$^{2}$}

\begin{document}

\maketitle

{\footnotesize 
\noindent $^{~1}$Université Paris Diderot -- Paris 7, 
Mathématiques, 
 case 7012, F--75205 Paris
Cedex 13, France
\\
\noindent e-mail:
\texttt{comets@math.univ-paris-diderot.fr}

\noindent $^{~2}$Department of Statistics, Institute of Mathematics,
 Statistics and Scientific Computation, University of Campinas --
UNICAMP, rua S\'ergio Buarque de Holanda 651,
13083--859, Campinas SP, Brazil\\
\noindent e-mails: \texttt{\{gallesco,popov,marinav\}@ime.unicamp.br}

}

\begin{abstract}
 Let $\T_n$ be the cover time of two-dimensional discrete torus $\Z^2_n=\Z^2/n\Z^2$.
We prove that $\IP[\T_n\leq \frac{4}{\pi}\gamma n^2\ln^2 n]=\exp(-n^{2(1-\sqrt{\gamma})+o(1)})$
for $\gamma\in (0,1)$.
One of the main methods used in the proofs is the decoupling 
of the walker's trace into independent excursions by means of
 soft local times.
\\[.3cm]\textbf{Keywords:} soft local time, hitting time, 
simple random walk
\\[.3cm]\textbf{AMS 2000 subject classifications:}
Primary 60G50, 82C41. Secondary 60G55.
\end{abstract}

\section{Introduction and results}
\label{s_introres}
Let $(X_t, t=1,2,3,\ldots)$ be a discrete-time simple random walk on the 
two-dimensional discrete torus $\Z^2_n=\Z^2/n\Z^2$. 
 Define the entrance time
to the site $x\in \Z^2_n$ by
\begin{equation}
\label{eq_defTx}
T_n(x) = \min\{t\geq 0: X_t=x\},
\end{equation}
and the \emph{cover time} of the torus by
\begin{equation}
\label{eq_defT}
\T_n = \max_{x\in \Z^2_n} T_n(x),
\end{equation}
that is, $\T_n$ is the first instant of time when all the sites
of the torus were already visited by the walk. 

The analysis of cover time by the planar random walk was suggested in~\cite{Wilf}
under the picturesque name of ``white screen problem'', and was soon after popularized 
 in the probabilistic community \cite[Chapter~7]{AldousFill}. 
We refer to~\cite{DemboSF} for a substantial survey on cover
 times, and to~\cite{ShiB} for a short account with a focus on exceptional points.
Besides being an appealing fundamental question, the study of cover time 
is of primer interest for performance evaluation of broadcast procedures in random networks,
see e.g.~\cite{KahnKLV}.

Not only natural, 
the two-dimensional model is also more difficult
than its higher-dimensional counterparts. 
This is because dimension two is critical for the walk,
resulting in strong correlations. 
To illustrate the dimension-based comparison, 
observe that 
very fine results are available for $d\geq 3$, see e.g.~\cite{B13} and references
therein, and also~\cite{GH12} where a closely related continuous problem
was studied. In contrast, in two dimensions the first-order asymptotics
of the cover time was completed only recently, 
after a series of intermediate steps over a decade of efforts.
 In~\cite{DPRZ04} it was proved that 
\begin{equation}
\label{conv_Tn}
 \frac{\T_n}{n^2\ln^2 n} \to \frac{4}{\pi} 
\text{ in probability, as $n\to\infty$}.
\end{equation}
More rough results, without the precise constant,
can be obtained using the Matthews' method~\cite{M88}.
The result~\eqref{conv_Tn} was then refined in~\cite{D12};
in the same paper it was suggested that $\sqrt{\T_n/2n^2}$ should
be around~$\sqrt{2/\pi}\ln n - c\ln\ln n$ for a \emph{positive}
constant~$c$
 (observe that~\eqref{conv_Tn}
means that $\sqrt{\T_n/2n^2}=\big(\sqrt{2/\pi}+o(1)\big)\ln n$).
This can be seen as a step towards the conjecture of~\cite{BZ09} 
 that $\sqrt{\T_n/n^2}$
should be tight around its median and nondegenerate. 
Such fine properties should be related to the fine structure of 
\emph{late points} of the walk, i.e., the sites 
that get covered only ``shortly'' before~$\T_n$. 
In spite of a very significant progress on this question
 achieved in~\cite{DPRZ06}, much remains to be 
discovered.

Now, we formulate our result on the deviations from below 
for the cover time: 
\begin{theo}
\label{t_lower}
 Assume that $\gamma\in(0,1)$. Then, for all $\eps>0$ we have
\begin{equation}
\label{eq_t_lower}
\exp\big(-n^{2(1-\sqrt{\gamma})+\eps}\big)\leq 
\IP\Big[\T_n\leq \frac{4}{\pi}\gamma n^2\ln^2 n\Big]
\leq \exp\big(-n^{2(1-\sqrt{\gamma})-\eps}\big)
\end{equation}
for all large enough~$n$.
\end{theo}

It should be mentioned that in~\cite{BGM} it was proved
that it is exponentially unlikely to cover \emph{any} 
bounded degree graph in \emph{linear} (with respect
to the number of vertices) 
number of steps. In this paper, however, we are concerned with times 
which differ from the cover time only by a constant factor, and so we obtain only stretched exponential decay.

\begin{rmk}
\label{rem_boxes}
 In fact, in Section~\ref{s_upbound} we prove a bit more than 
the upper bound in~\eqref{eq_t_lower}. Namely, 
assume that $\gamma\in (0,1)$, fix an arbitrary
$\alpha \in (\sqrt{\gamma}, 1)$ and tile the torus~$\Z^2_n$
with boxes of size~$n^\alpha$. Then there exist $c=c(\alpha,\gamma)>0$,
$c'=c'(\alpha,\gamma)>0$,
such that, at the moment $\frac{4}{\pi}\gamma n^2\ln^2 n$,
there are at least~$cn^{2(1-\alpha)}$ boxes which are not completely
covered, with probability at least~$1-\exp(-c'n^{2(1-\alpha)})$.
\end{rmk}

For completeness, we also include the result on the deviations from 
the other side:
\begin{theo}
\label{t_upper}
 Assume that $\gamma > 1$. Then, for all $\eps>0$ we have
\begin{equation}
 \label{eq_t_upper}
n^{-2(\gamma-1)-\eps} 
\leq \IP\Big[\T_n\geq \frac{4}{\pi}\gamma n^2\ln^2 n\Big] 
\leq n^{-2(\gamma-1)+\eps}
\end{equation}
for all large enough~$n$.
\end{theo}
However, it should be noted that the proof of Theorem~\ref{t_upper} 
is not difficult once one has~\eqref{conv_Tn}, 
although, to the best of our knowledge, 
it did not appear in the literature explicitly in this form.

To see how the proof of Theorem~\ref{t_upper} 
can be obtained, observe first that we 
have for all $\beta>0$, $\eps>0$, all  large enough $n$  and all $x\in\Z^2_n$, 
\begin{align}
 \max_{y\in \Z^2_n}\IP_y\Big[T_n(x)\geq \frac{2}{\pi}\beta n^2\ln^2 n\Big]
 &\leq n^{-\beta+\eps}, \label{eq_p_ht<}\\
 \min_{\substack{y\in \Z^2_n \\  y\neq x}}\IP_y\Big[T_n(x)\geq \frac{2}{\pi}\beta n^2\ln^2 n\Big]
 &\geq n^{-\beta-\eps}.
\label{eq_p_ht>}
\end{align}
The estimate~\eqref{eq_p_ht<} is Lemma~3.3 of~\cite{DPRZ06};
in fact, it is straightforward to modify the proof of the same lemma
to obtain~\eqref{eq_p_ht>}.

Now, the second inequality in~\eqref{eq_t_upper} immediately
follows from~\eqref{eq_p_ht<} and the union bound. As for the 
first inequality, the strategy for achieving this lower bound
can be described in the following way: let the random walk evolve
freely almost up to the expected cover time so that, with good probability
there are still uncovered sites, and then choose any particular
uncovered site and make the walk avoid it till the end.
More precisely,
observe that, by~\eqref{conv_Tn}, for any fixed $\delta>0$ it holds that
\[
 \IP\Big[\T_n\geq \frac{4}{\pi}(1-\delta) n^2\ln^2 n\Big] \geq \frac{1}{2}
\]
for all~$n$ large enough; that is, at time $\frac{4}{\pi}(1-\delta) n^2\ln^2 n$
there is at least one uncovered site with probability at least $\frac{1}{2}$.
An application of~\eqref{eq_p_ht>} with $\beta=2(\gamma-1+\delta)$
concludes the proof of Theorem~\ref{t_upper}.

One can informally interpret~\eqref{eq_p_ht<}--\eqref{eq_p_ht>}
 in the following
way: hitting time of a fixed state has approximately exponential distribution
with mean $\frac{2}{\pi} n^2\ln n$. First, the convergence in~(\ref{conv_Tn}) 
agrees with the intuitive understanding
 that ``hitting times of different
sites should be roughly independent'', since the maximum of~$n^2$ i.i.d.\ 
exponential random variables with mean $\frac{2}{\pi} n^2\ln n$ is concentrated
around $\frac{4}{\pi} n^2\ln^2 n$. Moreover, the probability for the maximum of such 
r.v.'s to be larger by a factor $\gamma>1$ than this value is $n^{-2(\gamma-1)+o(1)}$. 
It is interesting to observe that, while Theorem~\ref{t_upper} still
agrees with this intuition, Theorem~\ref{t_lower} does not. Indeed, the probability
that the maximum of~$n^2$ i.i.d.\ 
exponential random variables with mean $\frac{2}{\pi} n^2\ln n$ is 
at most $\frac{4}{\pi}\gamma n^2\ln^2 n$ 
(where $\gamma\in (0,1)$) is of order
$(1-n^{2\gamma})^{n^2}\simeq \exp (-n^{2(1-\gamma)})$, 
which is not the actual order of magnitude obtained in Theorem~\ref{t_lower}. 
Thus, the behavior of the lower tails of the cover time reveals the 
fine dependence between hitting times of the different points on the torus.

To prove the upper bound in~\eqref{eq_t_lower}, we use the method
of \emph{soft local times} initially developed in~\cite{SLT},
where it was used to obtain strong decoupling inequalities 
for the traces left by random interlacements on disjoint sets.
This approach allows to simulate an adapted process on a general space~$\Sigma$
using a realization of a Poisson point process on~$\Sigma\times \R_+$. 
Naturally, one can use \emph{the same} realization of the Poisson process
to simulate \emph{several} different processes on~$\Sigma$, thus giving rise 
to a coupling of these processes. We do this to compare the excursions of
the random walk at different regions with the independent excursions,
that is, in some sense, we \emph{decouple} the traces of the random walk 
in different places, which of course makes things simpler.

Let us comment also on the large deviations for the cover time of the torus
 in dimension~$d\geq 3$. This question
was studied in~\cite{GH12} in the continuous setting, 
i.e., for the Brownian motion. Among other results, 
in~\cite{GH12} the many-dimensional counterparts of Theorems~\ref{t_lower}
(only the upper bound, 
by $\exp (-n^{d(1-\gamma)+o(1)})$) and~\eqref{t_upper} were obtained. 
We expect no substantial difficulties
in obtaining the same results for the random walk using the same
methods as in the present paper, except for the lower bound 
for the deviation probability from below,
since the approach of Section~\ref{s_lowbound}
 fails in higher dimensions. 

Notational convention: in the case when the 
starting point of the random walk is fixed, we indicate that in the subscript;
otherwise, the initial distribution of the random walk is considered 
to be uniform. The positive constants 
(not depending on~$n$ but possibly depending on the quantities, 
such as~$\gamma$ in Theorem~\ref{t_lower},
which are considered to be fixed)
are denoted by $c,c',c_1,c_3,c_4$ etc.
Also, it is convenient to view the random walks on the torus, 
simultaneously for all torus sizes~$n$, 
as the random walk on the full lattice observed modulo~$n \Z^2$.

%
%

\section{Soft local times}
\label{s_SLT}
In this section we describe the method  
of soft local times~\cite{SLT}, which is the key to 
the upper bound in~\eqref{eq_t_lower}.

First, we define the entrance time to a set $A\subset\Z^2_n$ by
\[
 T_n(A) = \min_{x\in A} T_n(x).
\]

 We write $x\sim y$ if~$x$ and~$y$ are neighbors
in the graph~$\Z^2_n$.
For $A\subset\Z^2_n$ let us define the (inner)
boundary of~$A$ by $\partial A = \{x\in A:\text{ there exists }y\notin A 
\text{ such that } x\sim y \}$. 


Next, for $A\subset \Z^2_n$ we define the entrance law to~$A$:
for $x\notin A$ and $y\in \partial A$ let
\begin{equation}
\label{df_H_A}
H_A(x,y) = \IP_x[X_{T_n(A)}=y]. 
\end{equation}

Let us now describe the method of soft local times, which allows 
us to compare excursions of the random walk with \emph{independent}
excursions. Let $A_1,\ldots,A_{k_0},A'_1,\ldots,A'_{k_0}\subset \Z^2_n$
be such that $A_j\subset A'_j$, $A_j\cap \partial A'_j=\emptyset$
for $j=1,\ldots,k_0$, and $A'_i\cap A'_j=\emptyset$ for $i\neq j$.
Let $A=\bigcup_{j=1}^{k_0}A_j$ and $A'=\bigcup_{j=1}^{k_0}A'_j$;
and assume that $\partial A'=\bigcup_{j=1}^{k_0}\partial A'_j$,
which implies also that
$\partial A=\bigcup_{j=1}^{k_0}\partial A_j$.

Now, suppose that we are only interested in the trace left by the random
walk on the set~$A$. Then, 
(apart from the initial piece of the trajectory until hitting~$\partial A'$
for the first time)
it is enough to know what are
the excursions of the random walk between the boundaries of~$A$ 
and~$A'$. To define these excursions, consider 
the following sequence of stopping times: 
\begin{align*}
 D_0 &= T_n(\partial A'),\\
 S_1 &= \min\{t> D_0 : X_t \in \partial A\},\\
 D_1 &= \min\{t> S_1 : X_t \in \partial A'\},
\end{align*}
and
\begin{align*}
 S_k &= \min\{t> D_{k-1} : X_t \in \partial A\},\\
 D_k &= \min\{t> S_k : X_t \in \partial A'\},
\end{align*}
for $k\geq 2$. 

We denote by $\Sigma_j$
the space of excursions between~$\partial A_j$ 
and~$\partial A'_j$; i.e., an element~$Z$ of this space is
a finite nearest-neighbor trajectory beginning
at a site of~$\partial A_j$ and ending on its first visit 
to~$\partial A'_j$. Denote also $\Sigma=\bigcup_{j=1}^{k_0}\Sigma_j$. 
The method of soft local times, 
as presented in~\cite{SLT}, provides a way of constructing 
the excursions between~$\partial A$ 
and~$\partial A'$ of the walk~$X$
using a Poisson point process
on $\Sigma \times \R_+$. 
To keep the presentation more clear and visual, 
we use another (in this case, equivalent) 
way of describing this approach, through
a \emph{marked} Poisson process 
 on~$\partial A\times \R_+$.

\begin{figure}
\centering \includegraphics[width=\textwidth]{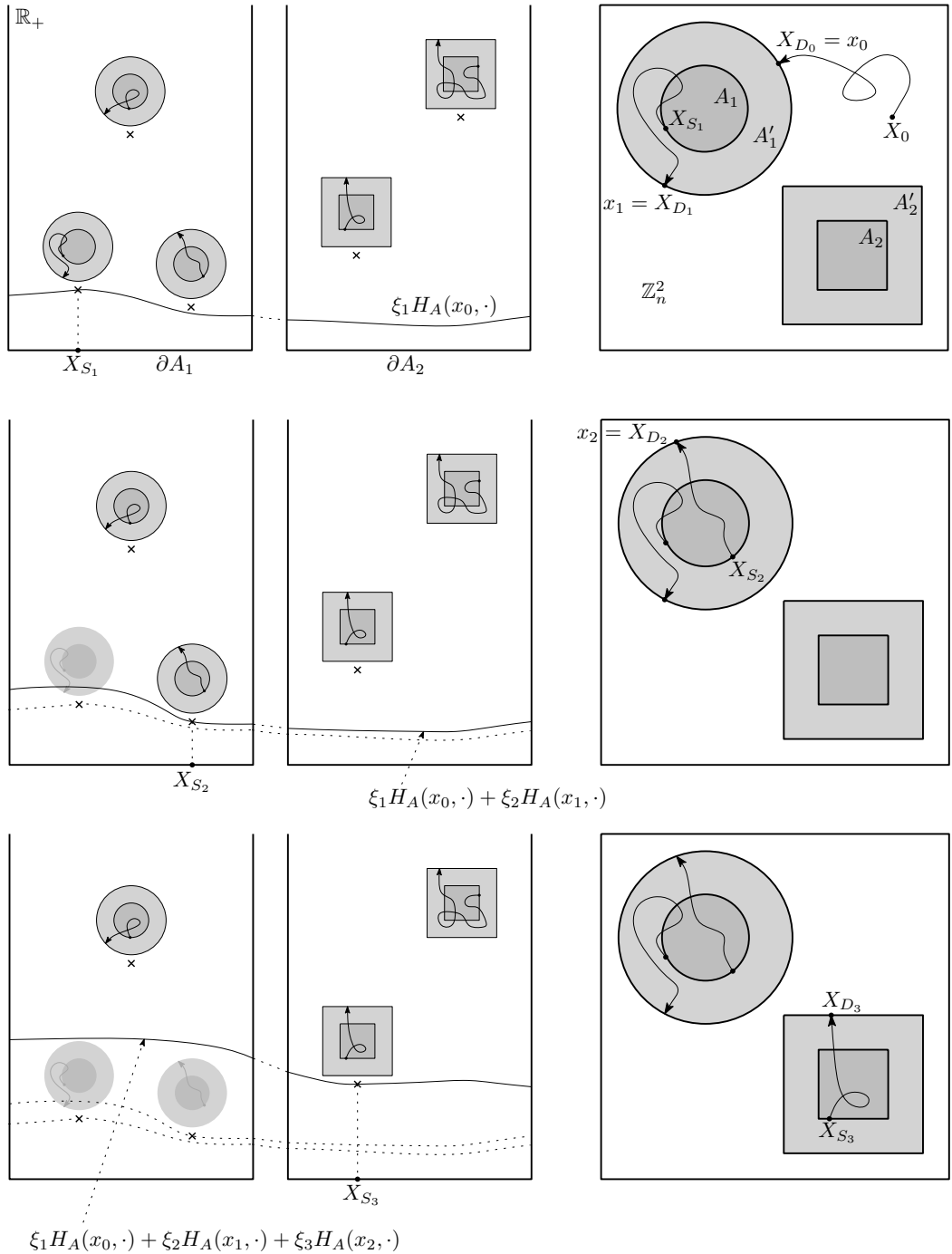}
  \caption{The construction of the excursions, the points are 
  represented with crosses, the marks are pictured above them. Observe that
we take the initial excursion (up to time~$D_0$) out
of consideration (even if $X_0\in A$).}
  \label{f:excursions}
\end{figure}

Denote by $Z_i=(X_{S_i}, \ldots, X_{D_i})$ the $i$th excursion
of~$X$ between~$\partial A$ and~$\partial A'$.
According to Section~4 of~\cite{SLT}, one can \emph{simulate}
 the sequence of excursions 
$(Z_i,i=1,2,3,\ldots)$
in the following way, see Figure~\ref{f:excursions}:
\begin{itemize}
 \item Consider a marked Poisson point process of rate~$1$
(with respect to 
(counting measure on~$\partial A$)$\times$(Lebesgue measure on~$\R_+$))
 on~$\partial A\times \R_+$, 
with independent marks.
\item These marks are the excursions of the simple random walk 
 starting at the corresponding site of~$\partial A$ and
stopped at the first visit to~$\partial A'$.
 \item At time~$D_0$ take $\xi_0>0$ such that
there is exactly one point of the Poisson process
on the graph of $\xi_0 H_A(x_0, \cdot)$ and
nothing below this graph, where $x_0=X_{D_0}$.
 \item The mark of this point is our first excursion~$Z_1$.
 \item Then, repeat the procedure, taking the 
graph of $\xi_0 H_A(x_0, \cdot)$ as ``$0$-level''.
\end{itemize}

Formally, on each ray $\{y\}\times\R_+$ (where $y\in\partial A$)
take an independent Poisson point process of rate~$1$. 
Together, these one-dimensional processes can be seen
as a random Radon measure
\[
 \eta = \sum_{\theta\in\Theta} \delta_{(z_\theta,u_\theta)}
\]
on the space~$\partial A\times \R_+$, where~$\Theta$
is a countable index set.
The marks $(\Psi_\theta, \theta\in\Theta)$ are 
independent excursions of the simple random walk, starting
at~$z_\theta$ and stopped at the first visit to~$\partial A'$.

Then (cf.\ Propositions~4.1 and~4.3 of~\cite{SLT}) define
\[
 \xi_{1} = \inf \big\{s \geq 0: \text{ there exists $\theta \in \Theta$ 
such that $s H_A(X_{D_0}, z_\theta) \geq u_\theta$}\big\},
\]
and
\[
  G_{1}(z) = \xi_{1} H_A(X_{D_0}, z), \text{ for $z \in \partial A$.}
\]
 Denote by $(z_1, u_1)$ 
 the a.s.\ unique pair in $\{(z_\theta, u_\theta)\}_{\theta \in \Theta}$
with $\xi_1 G_1( z_{1}) = u_{1}$, and let~$\Psi_1$ be the 
corresponding excursion.
Then, it holds that $\Psi_1$ is distributed as~$Z_1$ and the 
point process $\sum_{(z_\theta,u_\theta) \neq (z_1,u_1)} 
\delta_{(z_\theta, u_\theta - G_1(z_\theta))}$ is distributed as~$\eta$. 

 We can proceed iteratively to define 
$\xi_n$, $G_n$ and $(z_n,u_n)$ as follows
\begin{align*}
 \xi_{m} &=  \inf \big\{s \geq 0: 
\text{ there exists $(z_\theta, u_\theta) \notin \{(z_k, u_k)\}_{k=1}^{m-1}$}\\
&~~~~~~~~~~~~~~~~~~~~~\text{ such that }
 G_{m-1}(z_\theta) + sH_A(X_{D_{m-1}}, z_\theta) \geq u_\theta\big\},
\end{align*}
and
\[
  G_{m}(z) = G_{m-1}(z) + \xi_{m} H_A(X_{D_{m-1}}, z);
\]
then define $(z_m, u_m)$ as
 the unique pair $(z_\theta, u_\theta) \notin \{(z_k,u_k)\}_{k=1}^{m-1}$ 
with $G_m(z_\theta) = u_{\theta}$, and let~$\Psi_m$
be the corresponding excursion.
Then, one can show that $\xi_1,\xi_2,\xi_3,\ldots$ are i.i.d.\
random variables, exponentially distributed with parameter~$1$.
Also, it holds that the sequence of excursions
$(\Psi_{1}, \ldots, \Psi_{m})$ equals in law to $(Z_1, \ldots, Z_m)$, 
and these are independent from $\xi_1, \ldots, \xi_m$.
Also,
\[
 \sum_{\mathclap{\substack{\theta\in\Theta:\\
(z_\theta, u_\theta) \notin \{(z_k,u_k)\}_{k=1}^{m}}}} 
 \quad \delta_{(z_\theta, u_\theta - G_{m}(z_\lambda))}
\]
is distributed as~$\eta$ and independent of the above.
The function~$G_m$ is called the soft local time 
of the (excursion) process, the reason for this name 
is explained in Section~1.3 of~\cite{SLT}.
According to the above definitions, the soft local time 
in~$y$ up to $m$th excursion
is expressed as
\begin{equation}
\label{df_SLT}
  G_m(y) = \sum_{i=1}^{m} \xi_i H_A(X_{D_i}, y).
\end{equation}


\begin{figure}
\centering \includegraphics{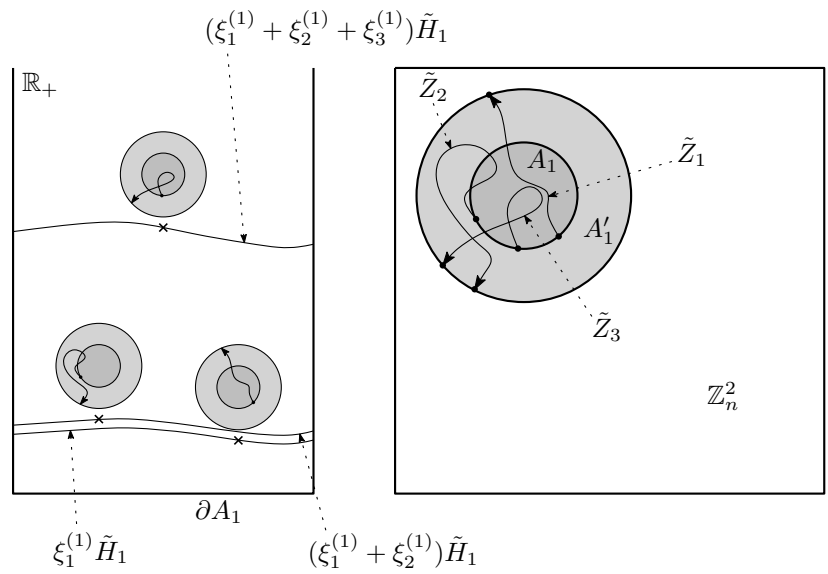}
  \caption{The construction of the i.i.d.\ excursions
between~$\partial A_j$ and~$\partial A'_j$. 
It is important to observe that the points of the
 Poisson process appear in different order 
in this construction
when compared to the corresponding excursions on Figure~\ref{f:excursions}
(note that we use \emph{the same} realization of the Poisson
process).}
  \label{f:independent_exc}
\end{figure}

We need to introduce some further notations.
Let us write $x\in Z$ when the excursion~$Z$
passes through~$x\in A$.
Consider any probability measure ${\tilde H}_j(\cdot)$
on~$\partial A_j$.
Let ${\tilde Z}_1^{(j)}, {\tilde Z}_2^{(j)}, {\tilde Z}_3^{(j)},\ldots 
\in \Sigma_j$ be a sequence of \emph{independent} elements of
the excursion space, chosen according to the following
procedure: take a starting point $x\in\partial A_j$ 
with probability~${\tilde H}_j(x)$, and then run the simple random
walk until it hits~$\partial A'_j$. Similarly 
to the previous construction of the excursions of the random walk~$X$,
we can simulate the sequence 
${\tilde Z}_1^{(j)}, {\tilde Z}_2^{(j)}, \ldots$ of independent excursions
in the same way,
and its soft local time in~$y$ up to time~$m$
equals
\begin{equation}
\label{df_SLT_indep}
  {\tilde G}^{(j)}_m(y) = {\tilde H}_j(y)\sum_{i=1}^{m} \xi^{(j)}_i,
\end{equation}
where $\big(\xi^{(j)}_1,\xi^{(j)}_2,\xi^{(j)}_3,\ldots\big)$
is another sequence of Exp(1) i.i.d.\ random variables.
For the construction of this sequence of independent excursions, 
we use \emph{the same realization} of the marked 
Poisson point process, thus creating a coupling of the sequence
of the excursions of~$X$ with~$k_0$ collections of i.i.d.\ excursions 
 (see Figure~\ref{f:independent_exc}). 
At this point we have to observe that the sequence
$(\xi_i, i\geq 1)$ is \emph{not} independent from the
collection of sequences $(\xi^{(j)}_i, i\geq 1, j=1,\ldots,k_0)$,
although this fact does not result in any major complications.

Let us denote
\[
 \sigma_1^{(j)} = \min\{i\geq 1: Z_i \in\Sigma_j\},
\]
and, for $m\geq 1$,
\[
 \sigma_{m+1}^{(j)} = \min\{i > \sigma_m^{(j)}: Z_i \in\Sigma_j\}.
\]
Then, we denote by $Z^{(j)}_i:=Z_{\sigma_i^{(j)}}$ the $i$th
excursion between~$\partial A_j$ and~$\partial A'_j$.
We also set $\psi_{j,t}=\max\{i: S_{\sigma_i^{(j)}} \leq t\}$, and
then denote by $\zeta_j(t)=\sigma_{\psi_{j,t}}^{(j)}$
the number of excursions between~$\partial A_j$ and~$\partial A'_j$
up to time~$t$ (possibly including the
last incomplete one), and by $\zeta(t)=\sum_{j=1}^{k_0}\zeta_j(t)$
the total number of excursions up to time~$t$.

For $j=1,\ldots,k_0$ and $b>a>0$ define the random variables
\begin{equation}
\label{df_Nj}
 N_j(a,b) = \#\{\theta\in\Theta :
  z_\theta\in\partial A_j, 
   a{\tilde H}(z_\theta)<u_\theta\leq b{\tilde H}(z_\theta) \}.
\end{equation}

It should be observed that the analysis of the soft local
times is considerably simpler  in this paper than in~\cite{SLT}.
This is because here the (conditional) entrance measures to~$A_j$
are typically very close to each other (as in~\eqref{eq_seau1} below).
That permits us to make \emph{sure} statements about the 
comparison of the soft local times for different processes in case
when the realization of the Poisson process in $\partial A_j\times \R_+$
is sufficiently well behaved, as e.g.\ in~\eqref{df_seau_bon} below.
\begin{lem}
\label{l_seau}
 Assume that the probability 
measures~$({\tilde H}_j,j=1,\ldots,k_0)$ are such that
for all $y\in \partial A'$, $x\in \partial A_j$, 
$j = 1,\ldots,k_0$, and some~$v\in(0,1)$, we have
\begin{align}
1-\frac{v}{3}\leq\frac{\IP_y[X_{T_n(A)}=x\mid
X_{T_n(A)}\in A_j]}{{\tilde H}_j(x)}&\leq 1+\frac{v}{3}.
\label{eq_seau1}  
\end{align}
Futhermore, define the events
\begin{align}
 U_j^{m_0} &= \big\{N_j(m,(1+v)m)< 2vm, 
\nonumber\\
&~~~~~~~~~~~(1-v)m < N_j(0,m) < (1+v)m, 
\text{ for all }m\geq m_0\big\}.
\label{df_seau_bon}
\end{align}
Then, for all~$j=1,\ldots,k_0$ it holds that
\begin{itemize}
 \item[(i)] $\IP[U_j^{m_0}]\geq 1-c_1 \exp(-c_2 vm_0)$, and
 \item[(ii)] on the event $U_j^{m_0}$ we have for all $m\geq m_0$
 \begin{align*}
  \{{\tilde Z}^{(j)}_1,\ldots,{\tilde Z}^{(j)}_{(1-v)m}\} &\subset
     \{Z^{(j)}_1,\ldots, Z^{(j)}_{(1+3v)m}\},\\
     \{Z^{(j)}_1,\ldots, Z^{(j)}_{(1-v)m}\}&\subset
       \{{\tilde Z}^{(j)}_1,\ldots,{\tilde Z}^{(j)}_{(1+3v)m}\}.
 \end{align*}
\end{itemize}

\end{lem}

\begin{proof}
Fix any $j_0\in \{1,\ldots,k_0\}$ and
observe that $N_{j_0}(a,b)$ has Poisson distribution 
with parameter~$b-a$.
It is then straightforward to obtain~(i) using the usual large deviation bounds.

To prove~(ii), fix $k\geq 1$ and let
\[
 y_{j_0}^{(k)} = \argmin_{y\in\partial A_{j_0}} \frac{G_k(y)}{{\tilde H}_{j_0}(y)}
\]
(with the convention $0/0=+\infty$).
We then argue that for all~$k\geq 1$ we \emph{always} have
\begin{equation}
\label{<(1+v)}
 \frac{G_k(y)}{{\tilde H}_{j_0}(y)} \leq (1+v) \frac{G_k(y_{j_0}^{(k)})}{{\tilde H}_{j_0}(y_{j_0}^{(k)})}
  \qquad \text{for all }y\in \partial A_{j_0}.
\end{equation}
Indeed, by~\eqref{eq_seau1} we have
\begin{align*}
 \frac{G_k(y)}{{\tilde H}_{j_0}(y)} &=
    \frac{1}{{\tilde H}_{j_0}(y)}\sum_{\ell=1}^{k}\xi_\ell H_A(X_{D_{\ell-1}}, y)\\
    &= \sum_{\ell=1}^{k}\xi_\ell\frac{\IP_{X_{D_{\ell-1}}}[X_{T_n(A)}=y\mid 
 X_{T_n(A)}\in A_{j_0}]}
     {{\tilde H}_{j_0}(y)} \IP_{X_{D_{\ell-1}}}[X_{T_n(A)}\in A_{j_0}]\\
    &\leq \frac{1+\frac{v}{3}}{1-\frac{v}{3}}\cdot
    \sum_{\ell=1}^{k}\xi_\ell\frac{\IP_{X_{D_{\ell-1}}}[X_{T_n(A)}=y^{(k)}_0
    \mid X_{D_{\ell-1}}
\in A_{j_0}]}
     {{\tilde H}_{j_0}(y_{j_0}^{(k)})} \IP_{X_{D_{\ell-1}}}[X_{T_n(A)}\in A_{j_0}]\\
     &\leq (1+v) \frac{G_k(y_{j_0}^{(k)})}{{\tilde H}_{j_0}(y_{j_0}^{(k)})},
\end{align*}
since $(1+\frac{v}{3})/(1-\frac{v}{3})\leq 1+v$ for $v\in (0,1)$.

\begin{figure}
\centering \includegraphics{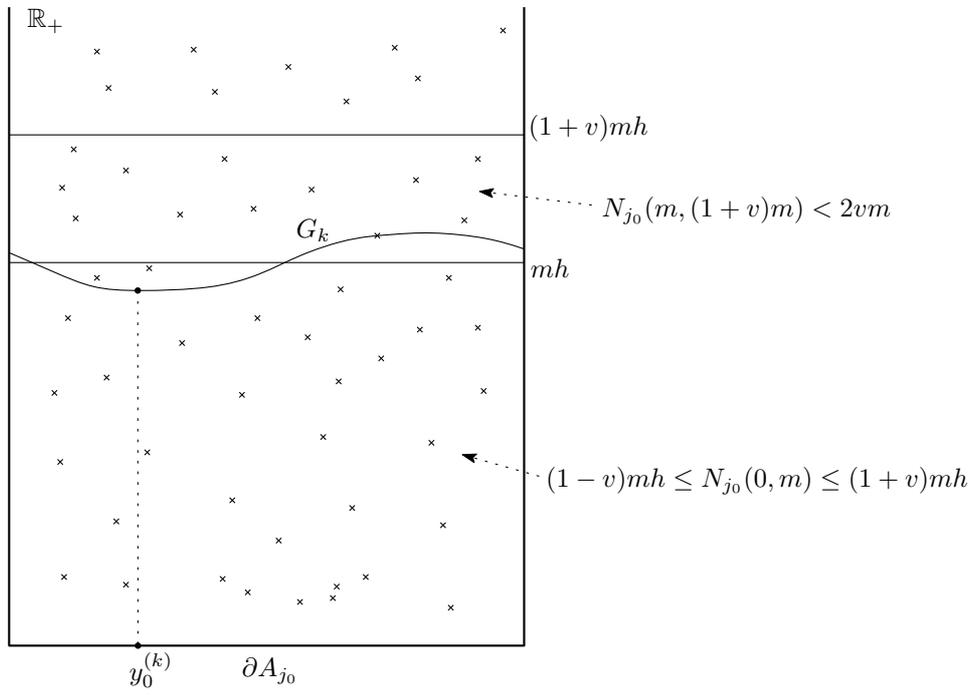}
  \caption{On the proof of Lemma~\ref{l_seau}. For simplicity, 
 here we assumed that ${\tilde H}_{j_0}\equiv h$ for a positive constant~$h$.}
  \label{f:compare_clt}
\end{figure}
Now, let~$m\geq m_0$, and abbreviate $k=\sigma^{(j_0)}_{(1-v)m}$.
 We then have
 $\frac{G_k(y_{j_0}^{(k)})}{{\tilde H}_{j_0}(y_{j_0}^{(k)})}\leq m$
(because otherwise, recall~\eqref{df_seau_bon}, we would
have more than~$(1-v)m$ points of the Poisson process below
the graph of~$G_k$), 
 and so, by~\eqref{<(1+v)}, $\frac{G_k(y)}{{\tilde H}_{j_0}(y)}\leq (1+v)m$
for all $y\in \partial A_{j_0}$ (see Figure~\ref{f:compare_clt}), which implies that 
\[
 \{Z^{(j)}_1,\ldots, Z^{(j)}_{(1-v)m}\} 
 \subset \{{\tilde Z}^{(j)}_1,\ldots,{\tilde Z}^{(j)}_{(1+3v)m}\}.
\]
 Analogously, for $k'=\sigma^{(j_0)}_{(1+3v)m}$ we must have
  $\frac{G_{k'}(y_0^{(k')})}{{\tilde H}_{j_0}(y_0^{(k')})}\geq m$
(because otherwise $\frac{G_{k'}(\cdot)}{{\tilde H}_{j_0}(\cdot)}$ 
would lie strictly below $(1+v)m$, and we would have
$N_j(0,(1+v)m)<(1+3v)m$), so
\[
 \{{\tilde Z}^{(j)}_1,\ldots,{\tilde Z}^{(j)}_{(1-v)m}\} \subset
     \{Z^{(j)}_1,\ldots, Z^{(j)}_{(1+3v)m}\},
\]
which concludes the proof of Lemma~\ref{l_seau}.
\end{proof}

%
%
%

\section{Proof of Theorem~\ref{t_lower}}
\label{s_proof_lowertail}


The proof is divided into two parts. First, in Section~\ref{s_upbound} we
use the method of soft local times to prove the second inequality in~\eqref{eq_t_lower}.
Then, in order to prove the first inequality 
in~\eqref{eq_t_lower} we present a particular strategy
for the walk, that assures that the torus will be covered with a not-too-small
probability by time $\frac{4}{\pi}\gamma n^2\ln^2 n$.

\subsection{Upper bound}
\label{s_upbound}

Note that for any fixed $x\in \Z^2_n$ 
there is a natural bijection of~$\Z^2_n$
and $[1,n]^2\subset \Z^2$ in such a way 
that~$x$ is mapped to 
$\big(\lceil\frac{n}{2}\rceil,\lceil\frac{n}{2}\rceil\big)\in\Z^2$.
Then, for $y\in \Z^2_n$ define~$\|y-x\|$ to be the Euclidean 
distance between $\big(\lceil\frac{n}{2}\rceil,\lceil\frac{n}{2}\rceil\big)$ 
and the image of~$y$, and we define also $\|y-x\|_1$ and~$\|y-x\|_\infty$
to be the~$\ell_1$ and the~$\ell_\infty$ distances correspondingly.
For $r < \frac{n}{2}$ we then define the discrete ball $B(x,r)\in\Z^2_n$
as the set of sites which are mapped by this bijection to the 
Euclidean ball of radius~$r$ centered 
in~$\big(\lceil\frac{n}{2}\rceil,\lceil\frac{n}{2}\rceil\big)$.



Define excursions between the balls $B(0,r)$ and $B(0,R)$ 
as in Section~\ref{s_SLT} (with $A_1=B(0,r)$, $A'_1=B(0,R)$,
$k_0=1$).

Now, we need to control the time it takes to complete the $j$th
excursion (see Lemma~3.2 of~\cite{DPRZ06}):
\begin{lem}
\label{l_number_exc}
 There exist $\delta_0>0$, $c>0$ such that if $r< R\leq \frac{n}{2}$
and $\delta\leq \delta_0$ with $\delta\geq 6c_1 (\frac{1}{r}+\frac{r}{R})$,
we have for all $x_0\in\Z^2_n$
\begin{equation}
\label{eq_number_exc}
\IP_{x_0}\Big[
D_j \leq (1+\delta)\frac{2n^2\ln\frac{R}{r}}{\pi}j\Big] \geq
 1-\exp\Big(-\frac{c\delta^2\ln\frac{R}{r}}{\ln\frac{n}{r}}j\Big).
\end{equation}
\end{lem}

Next, let us obtain the following consequence of Lemma~\ref{l_seau}:
\begin{lem}
\label{l_couple_cover}
Let $0<r_n<R_n<n/3$ be such that $r_n\geq \frac{n}{\ln^h n}$ for some~$h>0$.
Then for any $\phi\in(0,1)$, there exists $\delta>0$ such that
if ${\tilde H}$ is a probability measure on~$\partial B(0,r_n)$
with
\begin{equation}
\label{good_tildeH}
 \sup_{\substack{z\in \partial B(0,R_n)\\y\in \partial B(0,r_n)}}
\Big|\frac{H_{B(0,r_n)}(z,y)}{{\tilde H}(y)}-1\Big| < \delta
\end{equation}
then, as $n\to \infty$,
\begin{equation}
\label{eq_couple_cover1}
\IP\big[\text{there exists }y\in B(0,r_n) \text{ such that }
y\notin {\tilde Z}_j \text{ for all }j\leq k_0(n)\big] \to 1,
\end{equation}
where ${\tilde Z}_1,{\tilde Z}_2, {\tilde Z}_3,\ldots$ 
are i.i.d.\ excursions between~$\partial B(0,r_n)$
and $\partial B(0,R_n)$ with entrance measure~${\tilde H}$,
and $k_0(n)=2\phi\frac{\ln^2 R_n}{\ln R_n/r_n}$. 
\end{lem}

\begin{proof}
 Lemma~\ref{l_seau} implies that
one can choose a small enough~$\delta>0$ in such a way that one may couple
 the independent excursions with 
the excursion process $Z_1,Z_2,Z_3,\ldots$ of the random walk~$X$ on~$\Z^2_n$
so that
\[
 \{{\tilde Z}_1,\ldots, {\tilde Z}_{k_0(n)}\}
 \subset \{Z_1,\ldots,Z_{(1+\delta')k_0(n)}\}
\]
 with probability converging to~$1$ with~$n$, where~$\delta'>0$
is such that $(1+\delta')\phi<1$. 
Now, choose~$b$ such that $(1+\delta')\phi<b<1$
and observe that 
Theorem~1.2 of~\cite{DPRZ06} implies that a fixed
ball with radius at least $\frac{n}{\ln^h n}$ will not be completely 
covered up to time~$\frac{4}{\pi}bn^2\ln^2 n$ with probability
converging to~$1$. 
Together with Lemma~\ref{l_number_exc} this
implies that
\[
 \IP[B(0,r_n)\text{ is not completely covered by }
\{Z_1,\ldots,Z_{(1+\delta')k_0(n)}\}] \to 1
\]
 as $n\to\infty$, and this completes the proof of~\eqref{eq_couple_cover1}.
\end{proof}

We continue the proof of the upper bound in Theorem~\ref{t_lower}.
Fix an arbitrary $\alpha\in(\sqrt{\gamma},1)$, and let us 
denote
\[
s_n =\frac{n}{\lfloor n^{1-\alpha}\rfloor},\phantom{**}  k_n=\lfloor n^{1-\alpha}\rfloor^2.
\]
Let us tile the (continuous) torus $\R^2_n:=\R^2/n\Z^2$
with~$k_n$ squares with side~$s_n $.
Let us enumerate the squares in some way, and let $x'_1,\ldots,x'_{k_n}$ be
the sites at the centers of these squares. 
We then consider some isometric immersion of the torus~$\Z^2_n$
into $\R^2_n$, and denote by $x_1,\ldots,x_{k_n}\in \Z^2_n$
the (discrete) sites closest to $x'_1,\ldots,x'_{k_n}\in\R^2_n$.

Fix a small enough~$b\in(0,1/3)$ (to be specified later), and define
$A_j=B(x_j,bs_n )$, $A'_j=B(x_j,s_n /3)$;
also, as before, set $A=\bigcup_{j=1}^{k_n}A_j$ 
and $A'=\bigcup_{j=1}^{k_n}A'_j$. We construct 
the excursions of the random walk~$X$ between~$\partial A_j$ 
and $\partial A'_j$,
$j=1,\ldots,k_n$, as in Section~\ref{s_SLT}.
Then, fix any site $z_0\notin A'$
and define ${\tilde H}_j(x)=\IP_{z_0}[X_{T_n(A_j)}=x]$.

We need to show that the entrance measures to $A_j$, $j=1,\ldots,k_n$,
are ``almost equal to ${\tilde H}_j$'' 
on the boundary of each ball, if the parameter~$b$
are suitably chosen:
\begin{lem}
\label{l_flat_entrance}
 For any $\eps>0$ we can choose~$b\in (0,1/3)$ in such a way that 
for all $y\in \partial A'$, $x\in \partial A_j$, 
$j = 1,\ldots,k_n$, we have
\begin{align}
1-\eps\leq\frac{\IP_y[X_{T_n(A)}=x\mid
X_{T_n(A)}\in A_j]}{{\tilde H}_j(x)}&\leq 1+\eps,
\label{eq_flat_entrance1} 
\end{align}
\end{lem}
\begin{proof}
This fact easily follows e.g.\ from Lemma~2.2 of~\cite{DPRZ06}: one
can use conditioning on the position of the walk upon 
hitting $B(x_j,R)$ for a suitably chosen~$R$, and then use~(2.11)
of \cite{DPRZ06}.
\end{proof}

As in Section~\ref{s_SLT}, 
we denote by~$\zeta_j$ be the number of excursions of~$X$ 
between~$\partial A_j$ and $\partial A'_j$
up to time $\frac{4}{\pi}\gamma n^2\ln^2 n$, and
let $\zeta=\zeta_1+\cdots+\zeta_{k_n}$ be the total number 
of excursions.

 Let~$\gamma'$ be such that $\gamma<\gamma'<\alpha^2$.
Define the event
\[
 \Lambda_1 
= \Big\{\zeta \leq \frac{2\gamma' k_n \ln^2 n}{|\ln (3b)|}\Big\}
\]
(recall that $k_n$ is approximately $n^{2(1-\alpha)}$).
\begin{lem}
\label{l_total_excurs}
There is $c>0$ such that 
\begin{equation}
\label{eq_total_excurs}
 \IP[\Lambda_1] \geq 1-\exp\big(-c k_n \ln^2 n\big).
\end{equation}
\end{lem}
\begin{proof}
 It is tempting to write that the total number of excursions 
should have the same law as the number of excursions between~$B(0,bs_n )$
and $B(0,s_n /3)$ in~$\Z^2_{s_n }$
(if so, an application of Lemma~\ref{l_number_exc} would do the job). 
In the continuous setting this would 
work well, but, unfortunately, $s_n $ is not necessarily integer
which makes the above-mentioned equality in law formally false.

So, we proceed in the following way. First, by CLT one can obtain that 
there exists $c_1=c_1(b)>0$ such that $\IP_x[X_{s_n ^2}\in A]\geq c_1$
for all $x\in\Z^2_n$. This implies that 
\begin{equation}
\label{bounded_Eexp}
 \IE_x \exp\Big(\frac{T_n(A)}{s_n ^2}\Big) \leq c_2.
\end{equation}
Then, to find an upper bound on $\max_x \IE_xT_n(A)$, we can first approximate
the random walk with the Brownian motion by means of the multidimensional
version (Theorem~1 of~\cite{E89}) of the KMT strong approximation
theorem~\cite{KMT}, and then use Lemma~2.1 from~\cite{DPRZ04} together
with~\eqref{bounded_Eexp} to obtain the following fact:
for any~$\delta \in (0,\gamma'-\gamma)$ one can choose small enough~$b$
in such a way that
\begin{equation}
\label{bounded_E}
 \max_x \IE_xT_n(A) \leq \frac{2}{\pi} (\gamma+\delta) s_n ^2 |\ln(3b)|.
\end{equation}
The rest of the proof goes exactly in the same way as the proof of Lemma~3.2
(the relation (3.19) there) in~\cite{DPRZ06}.
\end{proof}

Next, fix~$\gamma''$ in such a way that $\gamma'<\gamma''<\alpha^2$.
If we had at least $\frac{\gamma'}{\gamma''}k_n$ 
balls among~$(A_1,\ldots,A_{k_n})$ with the 
corresponding number of excursions more than
$\frac{2\gamma''\ln^2 n}{|\ln (3b)|}$ in each of them, 
then the total number of excursions~$\zeta$ 
would be strictly greater than 
$\frac{2\gamma' k_n \ln^2 n}{|\ln (3b)|}$, 
so the event~$\Lambda_1$ would not occur. Thus, 
\begin{equation}
\label{many_small_exc}
 \text{on~$\Lambda_1$ we have that }\sum_{j=1}^{k_n} 
 \1{\zeta_j\leq \textstyle\frac{2\gamma''\ln^2 n}{|\ln (3b)|}} 
   \geq \Big(1-\frac{\gamma'}{\gamma''}\Big)k_n,
\end{equation}
i.e., on the event~$\Lambda_1$ the number of places where we have not too
many excursions is of order~$k_n$. 

Now, choose~$v>0$ in such a way that $(1+2v)\gamma'\alpha^{-2}<1$,
and assume that~$b$ is sufficiently small so that the hypothesis 
of Lemma~\ref{l_seau} holds on~$\Z^2_{s_n }$ for $r=bs_n $,
$R=s_n /3$ (Lemma~\ref{l_flat_entrance} assures that we can choose
such~$b$). 
Denote
\[
 \ell_1 := \frac{\gamma''\alpha^{-2}\ln^2 s_n }{|\ln (3b)|}, \qquad
\ell_2 := \frac{(1+3v)\gamma''\alpha^{-2}\ln^2 s_n }{|\ln (3b)|},
\]
and let ${\tilde Z}^{(j)}_1,{\tilde Z}^{(j)}_2,{\tilde Z}^{(j)}_3,\ldots$
be the independent excursions between~$A_j$ and~$A'_j$ obtained
using the coupling of Section~\ref{s_SLT}. 
Define the events $\Lambda_2^{(j)}=U_j^{\ell_1}$,
where~$U_j^{\ell_1}$ is the event in~\eqref{df_seau_bon},
and
\begin{align*}
 \Lambda_3^{(j)} &= \big\{\text{there exists }y\in A_j
 \text{ such that }y\notin {\tilde Z}^{(j)}_m
  \text{ for all }m\leq \ell_2\big\}.
\end{align*}

Observe that, by Lemmas~\ref{l_seau} and~\ref{l_couple_cover}, 
we have
\begin{equation}
\label{G2G3}
 \IP[\Lambda_2^{(j)}\cap \Lambda_3^{(j)}] \to 1 \qquad \text{ as }n\to \infty,
\end{equation}
for any~$j=1,\ldots,k_n$.

Next, choose ${\tilde \gamma} \in \big(\frac{\gamma'}{\gamma''},1\big)$, 
and define the event 
\begin{equation}
\label{df_Lambda4}
 \Lambda_4 = \Big\{\sum_{j=1}^{k_n}
  \1{\Lambda_2^{(j)}\cap \Lambda_3^{(j)}}
        \geq {\tilde\gamma}k_n+1\Big\};
\end{equation}
observe that the indicators in the above sum are i.i.d.\ random variables.
By~\eqref{G2G3}, for all large enough~$n$ it holds that (recall that 
$k_n=n^{2(1-\alpha)}(1+o(1))$)
\begin{equation}
\label{est_Lambda4}
 \IP[\Lambda_4] \geq 1-\exp(-c n^{2(1-\alpha)})
\end{equation}
But, taking~\eqref{many_small_exc} into account, we see that 
on $\Lambda_1\cap \Lambda_4$
at time $\frac{4}{\pi}\gamma n^2 \ln^2n$ we have
at least $\big({\tilde\gamma}-\frac{\gamma'}{\gamma''}\big)k_n$
balls among~$A_1,\ldots,A_{k_n}$ which are
not completely covered (observe that we have to exclude at most 
one ball that may have been crossed by the initial excursion
$(X_0,\ldots,X_{D_0})$; this is why we put ``$+1$'' in~\eqref{df_Lambda4}). 
This means that $\T_n>\frac{4}{\pi}\gamma n^2 \ln^2n$
on $\Lambda_1\cap \Lambda_4$, so the second
inequality in~\eqref{eq_t_lower} follows from~\eqref{est_Lambda4}
and Lemma~\ref{l_total_excurs}. \qed

\subsection{Lower bound}
\label{s_lowbound}
In this section, we prove the lower bound of~(\ref{eq_t_lower}). 
For this, we propose a simple strategy for the random walk to cover~$\Z_n^2$ 
before time $\frac{4}{\pi}\gamma n^2\ln^2 n$. We start with an informal discussion to outline the main ideas. 
We first divide the torus $\Z_n^2$ into $n^{2(1-\alpha)}$ 
boxes $B_1,\ldots,B_{n^{2(1-\alpha)}}$ of size $n^\alpha$ 
with $\alpha<\sqrt{\gamma}$. Since we want the 
random walk to cover the torus 
$\Z_n^2$ before time $t_{0}=\frac{4}{\pi}\gamma n^2\ln^2 n$, 
the natural strategy is to attempt to cover 
each box in time at most 
\[
r_n := \frac{t_{0}}{n^{2(1-\alpha)}}
=\frac{4\gamma}{\pi\alpha^2}n^{2\alpha}(\ln n^{\alpha})^2.
\] 
For this, we divide the time interval $[0, t_{0}]$ into time intervals 
$[(j-1)r_n , jr_n )$, for $j\in \{1,\dots, n^{2(1-\alpha)}\}$, 
and during each of them we force the random walk to spend most of the 
time 
in the box~$B_j$. In order 
to do this, we control the size of 
excursions of the random walk outside~$B_j$ 
and show that with probability greater 
than $\exp(-c\ln^{10}n)$ the time spent 
by the random walk in~$B_j$ is almost~$r_n $. Then, we show
 that the trace left by the random walk on~$B_j$ is not very different 
from the trace left on~$B_j$ by a
 random walk in a torus a bit larger than~$B_j$, with a not-too-small
probability (we invite the reader to look at Figure~\ref{f:cover_square}
to get an idea about how this is done). 
Since $\alpha<\sqrt{\gamma}$, this allows us to 
apply~\eqref{conv_Tn}
to conclude that, conditionally on the events mentioned above, 
with probability greater than a constant~$c'>0$ 
the random walk covers the box~$B_j$ 
during the time interval $[(j-1)r_n , jr_n )$. 
Finally, choosing $\alpha$ close 
enough to $\sqrt{\gamma}$ and applying the 
Markov property, we obtain the total cost
for this strategy that is at least $(c'\exp(-c\ln^{10}n))^{n^{2(1-\alpha)}}\geq 
\exp(-n^{2(1-\sqrt{\gamma})+\eps})$ for $\eps>0$.
\medskip

Now, let us start the proof. Let $\alpha\in (0,\sqrt{\gamma})$ and 
$N= \big\lceil \frac{n}{\lfloor n^{\alpha}\rfloor}\big\rceil$. 
We divide the torus $\Z_n^2$ into~$N^2$ 
boxes of size $\lfloor n^{\alpha}\rfloor$ 
(i.e., each box contains~$\lfloor n^{\alpha}\rfloor^2$ sites).  
The ``lower left" box is called $B_1$ (in this section the torus $\Z_n^2$ is identified with $[0,n)^2\subset \Z^2$) and the other 
boxes are positioned and enumerated following the arrows showed in Figure~\ref{f:enum_boxes} 
up to the box $B_{N^2}$. Observe that if~$n$ 
is not divisible by $\lfloor n^{\alpha}\rfloor$,
then the boxes $B_{jN}$, $B_{(j-1)N+1}$ on 
Figure~\ref{f:enum_boxes} have some area in common for
$j\in \{1,\dots,N\}$. The same is true for the boxes $B_j$, $B_{N^2-(j-1)}$ 
for $j\in \{1,\dots,N\}$.

\begin{figure}
\begin{center}
\includegraphics{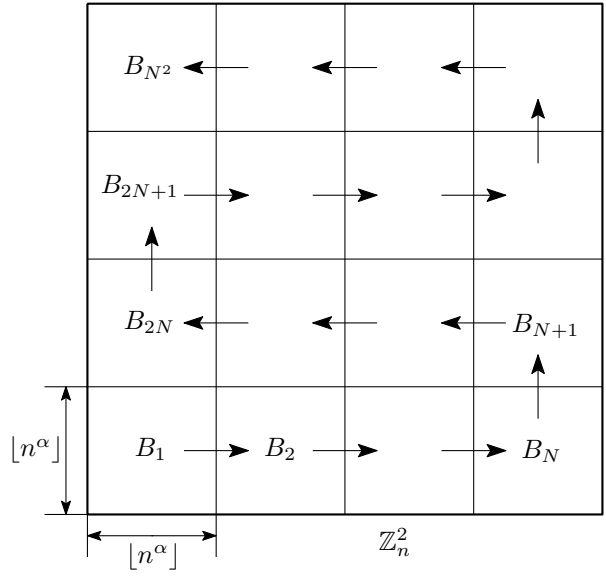}
\caption{Enumeration of the boxes $B_i$, $i\in \{1,\dots, N^2\}$.}
\label{f:enum_boxes}
\end{center}
\end{figure}

Let $\eta \in \big(0, 
\min\{1,\frac{1}{2}(\frac{\sqrt{\gamma}}{\alpha}-1)\}\big)$ 
and for all $i\in \{1,\dots, N^2\}$,  introduce the following sets
\[
B'_i=\{x\in \Z^2_{n}:\text{ there exists }y\in B_i
  \text{ such that }\|y-x\|_\infty\leq  \lfloor \eta n^{\alpha}\rfloor\}.
\]

Now, consider the torus $\Z^2_{\ell_n}$ where $\ell_n=2\lfloor \eta 
n^{\alpha}\rfloor +\lfloor n^{\alpha}\rfloor$ 
and fix a box~$B$ of size $\lfloor n^{\alpha}\rfloor$ ``centered" in it.
Let 
\[
{\tilde B} =\{x\in\Z^2_{\ell_n} : \text{ for all }y\in B,
  \|y-x\|_\infty\geq \lfloor \eta n^{\alpha}\rfloor\}
\]
be the ``boundary'' of the torus~$\Z^2_{\ell_n}$.
For all $i\in \N$, we consider the sequence 
$Y^{(i)}$ (independent of $X$) of i.i.d.\ 
random elements, where for each $i\geq 1$, 
\[
Y^{(i)}=\Big\{ Y^{(i)}_{j,x}, x\in {\tilde B}, j\geq 1 \Big\},
\]
and the $Y^{(i)}_{j,x}$ are independent random variables 
such that 
\[
\IP[Y^{(i)}_{j,x}=y]=H_B(x,y),
\]
where $H_B(x,\cdot)$ is the entrance law in~$B$ for the simple random walk 
on the torus $\Z^2_{\ell_n}$ starting from~$x$, 
similarly to~\eqref{df_H_A}.
Using the natural identification of the boxes $B'_i$ with $\Z^2_{\ell_n}$ and 
the boxes $B_i$ with $B$, 
each random element $Y^{(i)}$ will be viewed as a set of 
random variables indexed by $\partial B'_i$ and $j\geq 1$ and taking values in $B_i$.

Set $V_0=0$. For $i\in \{1,\dots, N^2\}$, 
we define inductively (see Figure~\ref{f:cover_square}):
\begin{align*}
\sigma_0^{(i)}&=V_{i-1},
\nonumber\\
\tau_0^{(i)}&=\inf \big\{ t\geq \sigma_0^{(i)}: X_t \in \partial B'_i\big\}
\end{align*}
(observe that for $i=1$ the value of~$V_{i-1}=V_0$ is set to be equal to~$0$,
and, for the next steps, see~(\ref{ITYE}) below)
and for all $j\geq 1$, define
\begin{align*}
\sigma_j^{(i)}&=\inf \Big\{ t\geq \tau_{j-1}^{(i)}: 
X_t=Y^{(i)}_{j,X_{\tau^{(i)}_{j-1}}}\Big\},\nonumber\\
\tau_j^{(i)}&=\inf \big\{ t\geq \sigma^{(i)}_j: X_t \in \partial B'_i\big\}.
\end{align*}
Let $\delta>0$ and recall that $r_n =\frac{4}{\pi}\gamma n^{2\alpha}\ln^2 n$. 
We also define 
\begin{equation*}
J_i=\inf\Big\{j\geq 0: \sum_{k=0}^j(\tau_k^{(i)}-\sigma_k^{(i)})
\geq \lfloor(1-\delta)r_n \rfloor\Big\}
\end{equation*}
and
\begin{align*}
\beta_i=\sigma_{J_i}^{(i)}+\lfloor(1-\delta)r_n \rfloor
- \sum_{k=0}^{J_i-1}(\tau_k^{(i)}-\sigma_k^{(i)}).
\end{align*}
Finally, we define 
\begin{equation}
\label{ITYE}
V_i=\inf\big\{t\geq \beta_i: X_t= w_i\big\}
\end{equation}
where $w_i$ is the lower left corner point of the box $B_{i+1}$.

\begin{figure}
\centering \includegraphics{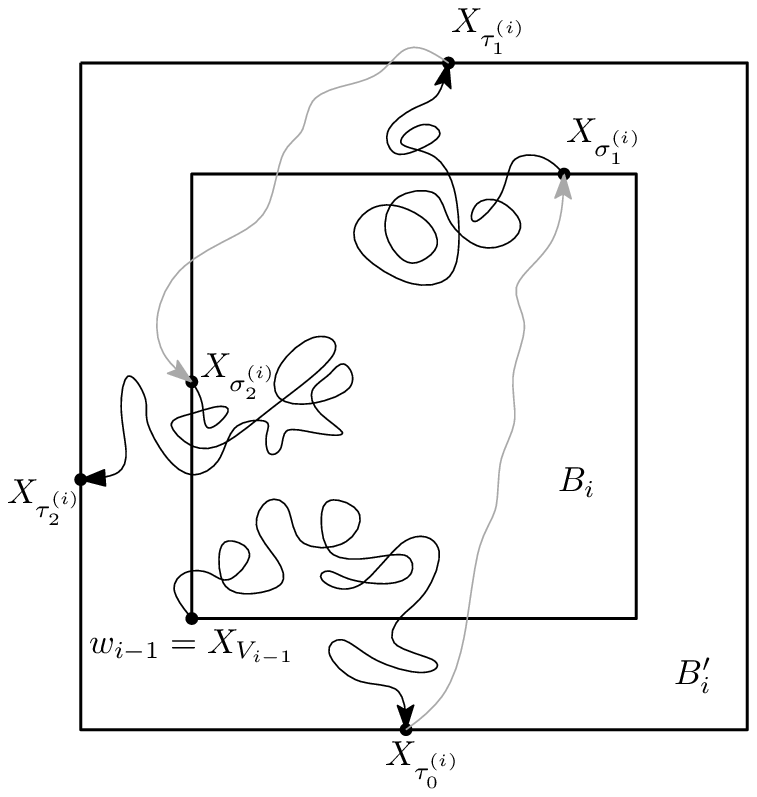}
  \caption{The strategy for covering the box $B_i$. We let the 
walk evolve freely until it hits the boundary of~$B'_i$. Then, 
we \emph{force} the walk to go rapidly to a random site
of $\partial B_i$ (this corresponds to the gray
parts of the trajectory). This random site is chosen according
to the entrance law to~$B_i$ \emph{as if} we had the torus~$\Z^2_{\ell_n}$
instead of the box~$B'_i$. This allows us to dominate the trace
of the random walk~${\hat X}$ on~$B\subset\Z^2_{\ell_n}$ by the trace
of the random walk~$X$ on~$B_i$.}
  \label{f:cover_square}
\end{figure}

By transitivity of the simple random walk on the torus $\Z_n^2$ we have that 
\begin{align}
\label{RTY3}
\IP\Big[\T_n\leq \frac{4}{\pi}\gamma n^2\ln^2 n\Big]=\IP_x\Big[\T_n\leq 
\frac{4}{\pi}\gamma n^2\ln^2 n\Big]
\end{align}
for all $x\in \Z^2_n$. So, in the rest of the proof we assume that $x=0$.

\medskip

Define $\S^{(i)}$ as the trace left by the excursions of the random 
walk~$X$ during the time intervals $[\sigma^{(i)}_j, 
\tau^{(i)}_j]$, $0\leq j< J_i$ and $[\sigma_{J_i}^{(i)},\beta_i]$. 
Define the events $M_i$, for $i\in \{1,\dots, N^2\}$ as
\begin{align*}
M_i
&=\Big\{J_i\leq \ln^6 n, B_i \subset \S^{(i)}\Big\}\cap 
\Big\{ \sigma^{(i)}_{j+1}
-\tau^{(i)}_{j}\leq \delta \frac{4}{\pi}\gamma 
\frac{n^{2\alpha}}{\ln^4 n}, 0\leq j< J_i\Big\}\nonumber\\
&\phantom{**}\cap 
\Big\{ V_i-\beta_i\leq \delta \frac{4}{\pi}\gamma 
\frac{n^{2\alpha}}{\ln^4 n}\Big\}.
\end{align*}
Observe that $\bigcap_{i=1}^{N^2}M_i$ is a desired strategy:
\begin{align}
\label{RTY2}
\Big\{\bigcap_{i=1}^{N^2}M_i\Big\} \subset \Big\{\T_n\leq \frac{4}{\pi}\gamma n^2\ln^2 n\Big\} .
\end{align}

For $i\in \{1,\dots,N^2\}$ we introduce the 
$\sigma$-fields $\G_{V_i}=\F_{V_i}\vee \sigma(Y^{(j)}, j\leq i)$, 
where $\F_{V_i}$ is the $\sigma$-field 
generated the random walk~$X$ until time~$V_i$. 
Conditioning iteratively by $\G_{V_i}$ for $i\in \{1,\dots,N^2\}$ and 
using the strong Markov property of $X$ (observe that $X$ still has the 
strong Markov property when conditioning by $\G_{V_i}$ since the random 
elements $Y^{(i)}$ are independent of~$X$), we obtain that
\begin{equation*}
\IP_0\Big[\bigcap_{i=1}^{N^2}M_i\Big]=\big(\IP_0[M_1]\big)^{N^2}.
\end{equation*} 
We will now estimate $\IP_{0}[M_1]$. For this, 
we introduce the $\sigma$-field $\HH$ 
generated by the random element $Y^{(1)}$ and by~$X$ within the time intervals $([\sigma^{(1)}_j, \tau^{(1)}_j], 0\leq j<J_1)$ and $[\sigma_{J_1}^{(1)}, \beta_1]$. 
Define also the events
\[
\Phi^{(1)}_j = \Big\{\sigma^{(1)}_{j+1}-\tau^{(1)}_{j}\leq \delta 
\frac{4}{\pi}\gamma \frac{n^{2\alpha}}{\ln^4 n} \Big\} 
\]
for $0\leq j<J_1$ and
\[
\Phi^{(1)}_{J_1} = \Big\{ V_1-\beta_1\leq \delta \frac{4}{\pi}\gamma 
\frac{n^{2\alpha}}{\ln^4 n}\Big\}.
\]
By definition of $M_1$ we have 
\begin{align}
\label{ERT}
\IP_{0}[M_1]
&=\IP_{0}\Big[J_1\leq \ln^6n, B_1\subset \S^{(1)}, 
\bigcap_{j=0}^{J_1} \Phi^{(1)}_j
\Big]\nonumber\\
&=\IE_{0}\Big[\1{J_1\leq \ln^6n, B_1\subset \S^{(1)}}\IP_{0}
\Big[\bigcap_{j=0}^{J_1}\Phi^{(1)}_j
\mid \HH\Big] \Big].
\end{align}
Now observe that, conditioned on~$\HH$, 
the events $\Phi^{(1)}_j$,
$0\leq j\leq J_1$, are independent. 
Further, in the time interval $[\tau^{(1)}_{j},\sigma^{(1)}_{j+1}]$, for $0\leq j<J_1$,
we have an excursion of~$X$ starting 
at point $X_{\tau^{(1)}_j}$ on~$\partial B'_1$ 
and ending at point $Y^{(1)}_{j,X_{\tau^{(1)}_j}}$ on $\partial B_1$. The last excursion in the time interval $[\beta_1, V_1]$ is conditioned to start from some point in $B'_1$ and to end at the lower left corner of $B_2$. Considering the process $S=(S_t)_{t\geq 0}$ 
which under the measure $P_x$ is a random walk on $\Z^2$ starting at $x$, we deduce that
\begin{align*}
\IP_{0}\Big[\bigcap_{j=0}^{J_1}\Phi^{(1)}_j
\mid \HH\Big]
&\geq \Big(\inf_{\substack{x\in  B'_1,\\ y\in \partial B_1}}
\IP_x[X_v=y]\Big)^{J_1}\nonumber\\
& \geq \Big(\inf_{\substack{x\in  B'_1,\\ y\in \partial B_1}}
P_x[S_v=y]\Big)^{J_1}
\end{align*}
with $v=\big\lfloor \delta \frac{4}{\pi}\gamma 
\frac{n^{2\alpha}}{\ln^4 n}\big\rfloor$ if 
$\big\lfloor \delta \frac{4}{\pi}\gamma 
\frac{n^{2\alpha}}{\ln^4 n}\big\rfloor$ and $\|x-y\|_1$ 
have the same parity (where $\|\cdot\|_1$ is the 1-norm on $\Z^2_n$) and 
$v=\lfloor \delta \frac{4}{\pi}\gamma \frac{n^{2\alpha}}{\ln^4 n}\rfloor-1$ otherwise.
Using the local central limit theorem 
(see e.g.\ Theorem~2.1.3 in~\cite{LL10}) 
and the fact that $\|x-y\|_1 \leq 4n^{\alpha}$ (recall that $\eta<1$), 
we obtain 
\begin{align*}
 \Big(\inf_{\substack{x\in B'_1,\\ y\in \partial B_1}}
P_x\Big[S_v=y\Big]\Big)^{J_1}
 \geq \exp\Big(-\frac{c_0 J_1\ln^4 n}{\delta \gamma}  \Big).
\end{align*}
for some constant $c_0>0$ and~$n$ large enough.
From (\ref{ERT}), we deduce
\begin{align}
\label{RTY}
\IP_{0}[M_1]\geq \exp\Big(-\frac{c_0\ln^{10} n}{\delta \gamma}  
\Big)\times \IP_{0}[J_1\leq \ln^6n, B_1\subset \S^{(1)}]
\end{align}
for~$n$ large enough.
Let us now bound from below the probability in the right-hand side of~(\ref{RTY}). 
We start by writing
\begin{align}
\label{Eqlb2}
\IP_{0}[J_1\leq \ln^6n, B_1\subset \S^{(1)}]
& \geq \IP_{0}[J_1\leq \ln^6n]-\IP_{0}[B_1\not\subset \S^{(1)}].
\end{align}
Now, let $\Q_x$ be the law of a simple 
random walk $\hat{X}$ on $\Z^2_{\ell_n}$ 
starting at~$x$ and define the random variables 
$\hat{\sigma}_j$, $\hat{\tau}_j$, $\hat{\beta}$, $\hat{J}$ and $\hat{\S}$ for $\hat{X}$ 
analogously to $\sigma^{(1)}_j$, $\tau^{(1)}_j$, $\beta_1$
$J_1$ and $\S^{(1)}$ for~$X$ ($B$ and~${\tilde B}$ play the role of~$B_1$
and $\partial B'_1$, correspondingly). 
Observe that by construction, the excursions of~$X$ during 
the time intervals $[\sigma^{(1)}_j, \tau^{(1)}_j]$ 
until time $\beta_1$ 
have the same law under $\IP_0$ as the excursions of $\hat{X}$ during the 
time intervals $[\hat{\sigma}_j, \hat{\tau}_j]$ 
until time $\hat{\beta}$ 
under $\Q_{x_0}$ where $x_0:=(\lfloor \eta n^{\alpha}\rfloor,\lfloor \eta n^{\alpha}\rfloor)$. 
Therefore, we have
\begin{align}
\label{WER}
\IP_{0}[J_1\leq \ln^6n, B_1\subset \S^{(1)}]
&\geq \Q_{x_0}[\hat{J}\leq \ln^6n]-\Q_{x_0}[B \not\subset \hat{\S}]\nonumber\\
& \geq \Q_{x_0}[\hat{J}\leq \ln^6n]-\Q_{x_0}[\T_{\ell_n}>(1-\delta)r_n ].
\end{align}
Using the fact that $\eta< \frac{1}{2}(\frac{\sqrt{\gamma}}{\alpha}-1)$ 
we can choose $\delta>0$ such that $\delta<1-\frac{\alpha^2(1+2\eta)^2}{\gamma}$, 
then by~\eqref{conv_Tn}
we obtain
\begin{equation}
\label{Eqlb3}
\Q_{x_0}[\T_{\ell_n}>(1-\delta)r_n ]\leq \frac{1}{4}
\end{equation}
for all~$n$ large enough.

Now let us show that $\Q_{x_0}[\hat{J}\leq \ln^6 n]\geq \frac{1}{2}$ 
for all large enough~$n$.
We first introduce the following event 
\[
\Lambda=\Big\{\text{there exists $j\in \{0,\dots, \hat{J}-1\}$ such that 
$\hat{\tau}_j-\hat{\sigma}_j\leq \frac{r_n }{\ln^6 n}$} \Big\}.
\]
Since $\hat{J}\leq \frac{r_n }{\lfloor \eta n^{\alpha}\rfloor}$ 
(indeed, as any excursion starts from $\partial B$ 
and ends at $\tilde{B}$, 
we need at least $\lfloor \eta n^{\alpha}\rfloor$ steps to complete it), 
we obtain by the Markov property
\begin{align}
\label{Eqlb}
\Q_{x_0}[\hat{J}> \ln^6 n]
&\leq \Q_{x_0}[\Lambda]\nonumber\\
&\leq \sum_{j=0}^{\lfloor r_n \lfloor \eta n^{\alpha}\rfloor^{-1}\rfloor-1}
\Q_{x_0}\Big[\hat{\tau}_j-\hat{\sigma}_j
\leq \frac{r_n }{\ln^6 n}\Big]\nonumber\\
&\leq \frac{r_n }{\lfloor \eta n^{\alpha}\rfloor}\sup_{x\in \partial B}
\Q_x\Big[\max_{t\leq r_n \ln^{-6 }n}\|\hat{X}_t\|_1\geq \eta n^\alpha\Big]\nonumber\\
& = \frac{r_n }{\lfloor \eta n^{\alpha}\rfloor}\sup_{x\in \partial B}
P_x\Big[\max_{t\leq r_n \ln^{-6 }n}\|S_t\|_1\geq \eta n^\alpha\Big].
\end{align}

Using item b) of Proposition~2.1.2 in~\cite{LL10}, 
we obtain that there exist positive constants~$c_1$ and~$c_2$ such that 
\begin{equation*}
\sup_{x\in \partial B}P_x\Big[\max_{t\leq r_n \ln^{-6 }n}\|S(t)\|_1
\geq\eta n^\alpha\Big]\leq c_1\exp(-c_2\eta^2\gamma^{-1}\ln^4 n).
\end{equation*}
Together with (\ref{Eqlb}) this implies that $\Q_{x_0}[\hat{J}> \ln^6 n]\to 0$ 
as $n\to \infty$ and therefore $\Q_{x_0}[\hat{J}\leq \ln^6 n]\geq \frac{1}{2}$ 
for~$n$ large enough.
Combining this fact with~(\ref{Eqlb2}), (\ref{WER}), 
and~(\ref{Eqlb3}) we obtain that 
\begin{equation}
\IP_{0}[J_1\leq \ln^6n, B_1\subset \S^{(1)}]\geq \frac{1}{4}
\end{equation}
for all $n$ large enough.
Finally, using (\ref{RTY}), (\ref{RTY2}) and (\ref{RTY3}) we deduce that
\begin{align*}
\IP\Big[\T_n\leq \frac{4}{\pi}\gamma n^2\ln^2 n\Big]
&\geq \Big( \frac{1}{4} \exp\Big(-\frac{c_0\ln^{10} n}{\delta \gamma}  
\Big)  \Big)^{N^2}\nonumber\\
&\geq  \exp\Big(-2n^{2(1-\alpha)}\Big(\frac{c_0}{\delta \gamma} 
\ln^{10}n+2\ln 2\Big)  \Big)
\end{align*}
for $n$ large enough.
Since $\alpha\in (0,\sqrt{\gamma})$ can be chosen arbitrarily 
close to $\sqrt{\gamma}$ we obtain the lower bound in Theorem~\ref{t_lower}.
\qed

\section*{Acknowledgments}
The authors thank the French-Brazilian 
program \textit{Chaires Fran\c{c}aises dans l'\'Etat
de S\~ao Paulo} which supported the visit of F.C.\ to Brazil.
S.P.\ and M.V. were partially supported by
CNPq (grants 300886/2008--0 and 301455/2009--0). 
The last three authors thank FAPESP (2009/52379--8)  
for financial support. F.C.~is partially supported by CNRS, UMR 7599 LPMA.

\end{document}